\documentclass[AMA,STIX1COL]{WileyNJD-v2}

\usepackage{hyperref}
\usepackage{amssymb}
\usepackage{amsmath}
\usepackage{bm}

\newcommand{\BR}{{\mathbb R}}
\newcommand{\BC}{{\mathbb C}}
\newcommand{\BN}{{\mathbb N}}

\newcommand{\D}{\bm D}
\newcommand{\X}{\bm x}
\newcommand{\XI}{\bm \xi}
\newcommand{\Y}{\bm y}

\newcommand{\e}{{\bf e}}

\newcommand{\cl}{C \kern -0.1em \ell}

\articletype{Article Type}%

\received{14 September 2020}
\revised{\today}

\begin{document}

\title{On Fundamental Solutions of Higher-Order Space-Fractional Dirac equations\protect\thanks{ORCID iD:~\href{https://orcid.org/0000-0002-9117-2021}{https://orcid.org/0000-0002-9117-2021}.}}

\author[1,2]{N. Faustino*}

\authormark{NELSON FAUSTINO}

\address[1]{\orgdiv{Faculty of Economics}, \orgname{University of Coimbra (FEUC)}, \orgaddress{\state{Coimbra}, \country{Portugal}}}

\address[2]{\orgdiv{Center for Research and Development in Mathematics and Applications (CIDMA)}, \orgname{University of Aveiro}, \orgaddress{\state{Aveiro}, \country{Portugal}}}

\corres{*Corresponding author\\ \email{\href{mailto:nelson@fe.uc.pt}{nelson@fe.uc.pt}} \\\href{mailto:nelson.faustino@ymail.com}{nelson.faustino@ymail.com}}

\presentaddress{Faculdade de Economia da Universidade de Coimbra 
	Av. Dias da Silva, 165,
	3004-512 Coimbra,
	Portugal}

\abstract[Summary]{
	
Starting from the pseudo-differential decomposition $\D=(-\Delta)^{\frac{1}{2}}\mathcal{H}$ of the Dirac operator $\displaystyle \D=\sum_{j=1}^n\e_j\partial_{x_j}$ in terms of the fractional operator $(-\Delta)^{\frac{1}{2}}$ of order $1$ and of the Riesz-Hilbert type operator $\mathcal{H}$ we will investigate the fundamental solutions of the space-fractional Dirac equation of L\'evy-Feller type
$$
 \partial_t\Phi_\alpha(\X,t;\theta)=-(-\Delta)^{\frac{\alpha}{2}}\exp\left( \frac{i\pi\theta}{2} \mathcal{H}\right)\Phi_\alpha(\X,t;\theta)
$$
involving the fractional Laplacian $-(-\Delta)^{\frac{\alpha}{2}}$ of order $\alpha$, with $2m\leq \alpha <2m+2$ ($m\in \BN$), and the exponentiation operator $\exp\left( \frac{i\pi\theta}{2} \mathcal{H}\right)$  as the hypercomplex counterpart of the fractional Riesz-Hilbert transform carrying the \textit{skewness parameter} $\theta$, with values in the range $|\theta|\leq \min\{\alpha-2m,2m+2-\alpha\}$.

Such model problem permits us to obtain hypercomplex counterparts for the fundamental solutions of higher-order heat-type equations $\partial_t F_M(x,t)=\kappa_M(\partial_x)^M F_M(x,t)$ $(M=2,3,\ldots)$ in case where the even powers resp. odd powers  $\D^{2m}=(-\Delta)^{m}$ ($M=\alpha=2m$) resp. $\D^{2m+1}=(-\Delta)^{m+\frac{1}{2}}\mathcal{H}$ ($M=\alpha=2m+1$) of $\D$ are being considered.

}



\maketitle



\section*{Dedicatory}

This paper is dedicated with esteem \textit{to Prof. Dr. rer. nat. habil. Klaus G\"urlebeck (Bauhaus-Universit\"at Weimar) on the occasion of his $65{th}$ birthday}.

\section{Introduction}

\subsection{State of Art}

Let $t>0$ and $\X,\XI\in \BR^n$. For real-valued functions $\varphi$ resp. $\psi$ with membership in the Schwarz space $\mathcal{S}(\BR^n)$ resp. the \textit{space of tempered distributions} $\mathcal{S}'(\BR^n)$, let
\begin{eqnarray}
\label{FourierTransform}
(\mathcal{F} \varphi)(\XI)=\int_{\BR^n} \varphi(\X)e^{-i \langle \X,\XI \rangle}d\X, & ~\varphi\in \mathcal{S}(\BR^n)~
\end{eqnarray}
be the Fourier transform over $\BR^n$ and
\begin{eqnarray}
\label{FourierInverse}(\mathcal{F}^{-1} \psi)(\X)=\frac{1}{(2\pi)^n}\int_{\BR^n} \psi(\XI)e^{i \langle \X,\XI \rangle}d\XI, & ~\psi\in \mathcal{S}(\BR^n)~ 
\end{eqnarray}
its inverse.

The kernel function $K_{\alpha,n}(\X,\tau)$ defined through the Fourier inversion formula $K_{\alpha,n}(\X,\tau)=(\mathcal{F}^{-1} e^{-\tau|\XI|^{\alpha}})(\X)$:
\begin{eqnarray}
\label{LevyDistribution}
K_{\alpha,n}(\X,\tau)= \frac{1}{(2\pi)^n}\int_{\mathbb{R}^n} e^{-\tau |\XI|^{\alpha}} e^{i\langle \X,\XI \rangle}d\XI &(~\Re(\tau)\geq 0~)
\end{eqnarray}
arises as the fundamental solution of the fractional heat equation of order $\alpha$  $(\alpha> 0)$ 
\begin{eqnarray}
	\label{FractionalHeat}\partial_t\Phi_\alpha(\X,t)=-(-\Delta)^{\frac{\alpha}{2}}\Phi_\alpha(\X,t).
\end{eqnarray}

Indeed, for $\tau=t$ the kernel function $K_{\alpha,n}(\X,t)$ solves eq. (\ref{FractionalHeat}) and satisfies the initial condition
$$\displaystyle K_{\alpha,n}(\X,0):=\lim_{t\rightarrow 0^+}K_{\alpha,n}(\X,t)=\delta(\X).$$

An effective way to reduce the right-hand side of (\ref{LevyDistribution}) to a one-dimensional integral is provided in \cite[p. 485, Lemma 25.1]{samko1993fractional} by the following Fourier inversion formula
\begin{eqnarray}
\label{FourierInversionRadial}
\int_{\BR^n} \phi(|\XI|)e^{i \langle \X, \XI \rangle} d\XI =\frac{(2\pi)^{\frac{n}{2}}}{|\X|^{\frac{n-2}{2}}} \int_{0}^{\infty} \phi(\rho) \rho^{\frac{n}{2}}J_{\frac{n}{2}-1}(\rho|\X|)d\rho
\end{eqnarray}
 underlying to radial functions $\psi(\XI)=\phi(|\XI|)$ (see also \cite[Theorem 3.3 of Chapter IV]{stein1971introduction}). Moreover, a simple calculation based on the change of variable $\rho\rightarrow \frac{\rho}{|\X|}$ on the right-hand side of (\ref{FourierInversionRadial}) shows that eq. (\ref{LevyDistribution}) can be cast in the form
\begin{eqnarray}
\label{LevyDistributionSpherical}
K_{\alpha,n}(\X,\tau)&=&\frac{1}{(2\pi)^{\frac{n}{2}}|\X|^{n}}\int_{0}^{\infty} e^{-\tau \left(\frac{\rho}{|\X|}\right)^{\alpha}}~\rho^{\frac{n}{2}}J_{\frac{n}{2}-1}(\rho)d\rho ~~~(\Re(\tau)\geq 0),
\end{eqnarray}
where $J_\nu$ denotes the $\nu-$th Bessel function (see \cite[Chapter 14]{george2013mathematical}).

We point out here that $K_{\alpha,n}(\X,\tau)$ corresponds to higher-dimensional generalization of the so-called L\'evy distribution of order $\alpha$ (cf.~\cite{bernstein2020fractional}). In the one-dimensional case, i.e. $\rho^{\frac{1}{2}}J_{-\frac{1}{2}}(\rho)=\sqrt{\frac{2}{\pi}}\cos(\rho)$ when $n=1$ (see~\cite[subsection 14.7]{george2013mathematical}), the integral representation (\ref{LevyDistribution}) has been tackled by renowned mathematicians such as F.~Bernstein \cite{bernstein1918fourierintegral}, P.~L\'evy \cite{levy1923application}, G.~Polya \cite{polya1923zeros}, W.R. Burwell \cite{burwell1924asymptotic} and D.V. Widder \cite{widder1979airy}.

In higher dimensions, the interest in studying the fundamental solution for fractional diffusion problems of type (\ref{FractionalHeat}) for values of $\alpha$ in the range $0<\alpha\leq 2$ has a vast literature (see e.g. \cite{caffarelli2007extension} and the references therein), especially when we are dealing with \textit{symmetric stable processes} of index $\alpha$ (cf. \cite{blumenthal1960stable,kolokoltsov2000symmetric}) in $\BR^n$ through the {\it positivity condition} $\displaystyle \varphi\geq 0 \Longrightarrow \exp(-t\left(-\Delta\right)^{\frac{\alpha}{2}})\varphi\geq 0$ of the diffusive semigroup $\displaystyle \left\{\exp(-t\left(-\Delta\right)^{\frac{\alpha}{2}})\right\}_{t\geq 0}$ of Markovian type.

In stark contrast, the cases where $\alpha>2$ -- corresponding to a fractional interpolation of higher-order heat equations (case of $\frac{\alpha}{2}\in \BN\setminus\{1\}$) -- are less understood (cf.~\cite{LiWong1993asymptotic}). In fact, the oscillator behavior of the Bessel functions appearing on the right-hand side of (\ref{LevyDistributionSpherical}) shows in turn that the {\it positivity condition} is no longer preserved for the  semigroup $\left\{\exp(-t\left(-\Delta\right)^{\frac{\alpha}{2}})\right\}_{t\geq 0}$ in case of $\alpha\geq 4$ and whence, the classical methods adopted for the second order heat equation do cannot be applied in this context. We refer to \cite{gazzola2010polyharmonic} for a wide overview of the problem and to \cite{ferreira2019eventual} for further advances beyond the biharmonic case (i.e. when $\frac{\alpha}{2}=2$) treated on the paper \cite{gazzola2008eventual}.

Accordingly to the literature, the optimal environment to deal with the biharmonic heat equation and fractional heat equation (\ref{FractionalHeat}) of order $\alpha>4$ can be provided by an exploitation of the notion of pseudo-Markov processes or even as composition of Brownian motions
or stable processes with Brownian motions. Those approaches has been popularized in the stochastic community since fundamental the papers of Hochberg (cf.~ \cite{hochberg1978signed}) and Funaki (cf. \cite{funaki1979probabilistic}), but we will not pursue these topics in the present paper.

On the other hand, we notice that the model problem involving the space-fractional heat equation (\ref{FractionalHeat}) only provides faithful generalizations for higher-order heat equations of even order (i.e. $\alpha=2m$, with $m\in\BN$). 
Also, on the multidimensional generalization of the L\'evy-Feller approach the lack of a closed form for the underlying generator semigroup is one of the major difficulties (cf.~\cite{hanyga2001multidimensional}). 

To the best of our knowledge, to derive faithful analogues of Airy like functions in higher dimensions, carrying higher-order heat operators of odd type (cf.~\cite{gorska2013higher}), one may consider the pseudodifferential reformulation of Dirac theory in the spirit of the seminal paper \cite{li1994clifford} of C.~Li, A.~McIntosh \& T.~Qian (see also ~\cite{mcintosh1996clifford} for a broad overview) in a way that the Dirac operator $\D$ is recasted in terms of the identity $\D=(-\Delta)^{\frac{1}{2}}\mathcal{H}$. Here and elsewhere, $\mathcal{H}$ stands for the Riesz-Hilbert transform.

The aim of this paper is to amalgate the study of the fundamental solutions of the polyharmonic heat-type $\partial_t+(-\Delta)^m$ ($m\in \BN$) and the higher-order Dirac-type operators $\partial_t\pm i\D^{2m+1}$ ($m\in \BN$) as well. To this end, we will investigate the fundamental solutions $\Phi_\alpha(\X,t;\theta)$
of the modified version of the space-fractional heat equation (\ref{FractionalHeat}) carrying the fractional Laplacian $-(-\Delta)^{\frac{\alpha}{2}}$ of order $\alpha$, with $2m\leq \alpha< 2m+2$ ($m\in \BN$):
\begin{eqnarray}
\label{FractionalDirac}
\left\{\begin{array}{lll}
\partial_t\Phi_\alpha(\X,t;\theta)=-(-\Delta)^{\frac{\alpha}{2}}\exp\left( \frac{i\pi\theta}{2} \mathcal{H}\right)\Phi_\alpha(\X,t;\theta)&,~\mbox{for} & (\X,t)\in \BR^n \times (0,\infty) \\ \ \\
\Phi_\alpha(\X,0;\theta)=\delta(\X) &,~\mbox{for} &\X \in \BR^n
\end{array}\right.
\end{eqnarray}

The evolution equation (\ref{FractionalDirac}) is build upon the Cauchy problem of L\'evy-Feller type considered by Gorenflo et al in \cite{gorenflo1999discrete}.
Here, the exponentiation operator 
$\exp\left( i\frac{\pi\theta}{2} \mathcal{H}\right)$ carrying the \textit{skewness parameter} $\theta$ in the range $|\theta|\leq \min\{\alpha-2m,2m+2-\alpha\}$
mimics the fractional Hilbert transform in optics introduced by Lohmann et al in \cite{lohmann1996fractional} (see also \cite[section 4]{bernstein2016fractional} and \cite[section 4]{bernstein2017fractional}).

The main gain of our model problem formulation agaisnt the multidimensional approaches available on the literature relies on the replacement of the parametrized measure on the unit sphere of $\BR^n$ -- hard to compute, in general (cf.~\cite{nolan2001estimation})-- by a rotation action in the plane parametrized by $\theta \mapsto \exp\left( \frac{i\pi\theta}{2} \mathcal{H}\right)$ -- easy to describe through the Fourier multiplier reformulation of $\mathcal{H}$ as $\mathcal{H}=\mathcal{F}^{-1} \frac{-i\XI}{|\XI|} \mathcal{F}$ (cf.~\cite[section 4]{bernstein2016fractional}).

\subsection{Layout of the paper and main results}

Let us summarize the layout of the paper: In Section \ref{Preliminaries} we establish the basic facts on Clifford algebras, Dirac operators and fractional Riesz-Hilbert-Type transforms following S.~Bernstein's footsteps (cf.~\cite{bernstein2016fractional,bernstein2017fractional}).
In Section \ref{MellinBarnesSection} we will obtain a Mellin-Barnes representation for the kernel function $K_{\alpha,n}(\X,\tau)$ in the spherical form (\ref{LevyDistributionSpherical}) and in Section \ref{ProofMainResults} we use the framework considered previously to investigate the fundamental solution associated to the Cauchy problem (\ref{FractionalDirac}).

Our first theorem to be proved in Subsection \ref{ProofMainResults1} shows that the fundamental solution $\Phi_\alpha(\X,t;\theta)$ of (\ref{FractionalDirac}) can be represented in terms of the kernel functions $K_{\alpha,n}(\X,te^{\pm i\frac{\pi\theta}{2}})$ represented through eq. (\ref{LevyDistribution}) [case of $\tau=te^{\pm i\frac{\pi\theta}{2}}$]:
\begin{theorem}\label{FundamentalSolutionDirac}
	The fundamental solution
	$\Phi_\alpha(\X,t;\theta)$ defined by the Cauchy-type problem (\ref{FractionalDirac}), where 
		\begin{eqnarray*}
		2m\leq \alpha< 2m+2 &\mbox{and} & |\theta|\leq \min\{\alpha-2m,2m+2-\alpha\} ~~~(m\in \BN)
	\end{eqnarray*} is equal to $$
\Phi_\alpha(\X,t;\theta)=\frac{1}{2}(I+\mathcal{H})K_{\alpha,n}(\X,te^{i\frac{\pi\theta}{2}})+\frac{1}{2}(I-\mathcal{H})K_{\alpha,n}(\X,te^{-i\frac{\pi\theta}{2}}).
	$$
\end{theorem}

\begin{remark}\label{thetaRemark2}
	We emphasize that the set of conditions
\begin{eqnarray*}
	2m\leq \alpha< 2m+2 &\mbox{and} & |\theta|\leq \min\{\alpha-2m,2m+2-\alpha\} ~~~(m\in \BN)
\end{eqnarray*}
lead to $-\frac{\pi}{2}\leq \frac{\pi \theta}{2}\leq \frac{\pi}{2}$, because of $\min\{\alpha-2m,2m+2-\alpha\}\leq 1$.
That is equivalent to say that $\Re(te^{\pm i\frac{\pi\theta}{2}} )=t\cos\left(\frac{\pi \theta}{2}\right)\geq 0$ is always fulfilled under the aforementioned conditions so that the kernel functions $K_{\alpha,n}(\X,te^{\pm i\frac{\pi\theta}{2}})$ appearing on the statement of {\bf Theorem \ref{FundamentalSolutionDirac}} are well defined.
\end{remark}

\begin{remark}\label{LiWongRemark}
In case of $\theta=0$ one can easily see that $
\Phi_\alpha(\X,t;0)$ coincides with the fundamental solution of the space-fractional heat equation (\ref{FractionalHeat}). In particular, from the choice $\alpha=2m$ one can moreover say that $
\Phi_{2m}(\X,t;0)$ approaches the fundamental solution of the higher-order heat operator $\partial_t+(-\Delta)^m$, already treated by X.~Li \& R.~Wong in \cite{LiWong1993asymptotic}.
\end{remark}

The second theorem to be proved in Subsection \ref{ProofMainResults1} is essentially a continuation of the proof of {\bf Theorem \ref{FundamentalSolutionDirac}} and combines analytic series expansions of Wright type (see \cite{kilbas2002generalized}) with the singular integral representation of the Riesz-Hilbert transform $\mathcal{H}$.  Its preparation starts in Section \ref{MellinBarnesSection} with the reformulation of the one-dimensional integral representation (\ref{LevyDistributionSpherical}) as a Mellin convolution integral that in turn permits us to represent $K_{\alpha,n}(\X,\tau)$ as a Wright series expansion
\begin{equation}
	\label{WrightSeriespq} {~}_1\Psi_1
	\left[\begin{array}{l|} (a_1,\alpha_1)   \\
		(b_1,\beta_1)
	\end{array} ~ \lambda \right]=\sum_{k=0}^\infty
	\dfrac{\Gamma(a_1+\alpha_1k)}{\Gamma(b_1+\beta_1
		k)}~\dfrac{\lambda^k}{k!},
\end{equation}
where $\lambda\in \BC$, $a_1,b_1\in \BC$ and $\alpha_1,\beta_1\in \BR\setminus\{0\}$.

In an overall view, such technique mimics the framework used by K.~G\'orska et al in the series of papers \cite{gorska2013exact,gorska2013higher} to derive, in one-dimensional case ($n=1$), \textit{signed L\'evy stable laws} and \textit{generalized Airy functions} as well. 

\begin{theorem}\label{FundamentalSolutionDiracSeries}
	Let $\Phi_\alpha(\X,t;\theta)$ be the kernel function determined in {\bf Theorem \ref{FundamentalSolutionDirac}.} for values of $\alpha$ and $\theta$ in the range
	\begin{center}
		$2m\leq \alpha< 2m+2$ and $ |\theta|\leq \min\{\alpha-2m,2m+2-\alpha\}$ ($m\in \BN$).
	\end{center} Then, we have the following:
	\begin{enumerate}
	\item $\Phi_\alpha(\X,t;\theta)$ admits the following singular integral representation 	\begin{eqnarray*}
		\Phi_\alpha(\X,t;\theta)&=&\displaystyle \frac{2^{1-n}}{\alpha \pi^{\frac{n}{2}}t^{\frac{n}{\alpha}}}~\Re\left(~e^{-\frac{i\pi \theta }{\alpha}\frac{n}{2}}~{~}_1\Psi_1
		\left[~\begin{array}{ll|} 
			(\frac{n}{\alpha},\frac{2}{\alpha})&  \\		
			(\frac{n}{2},1) & 
		\end{array} ~ -{\frac{|\X|^2\tau^{-\frac{2}{\alpha}}}{4}}e^{-\frac{i\pi \theta }{\alpha}} \right]\right)+\\
		&+&i\frac{2^{1-n}}{\alpha \pi^{\frac{n}{2}}t^{\frac{n}{\alpha}}}~\frac{\Gamma\left(\frac{n+1}{2}\right)}{\pi^{\frac{n+1}{2}}} P.V. \int_{\BR^n}~\Im\left(e^{-\frac{i\pi \theta }{\alpha}\frac{n}{2}}~{~}_1\Psi_1
		\left[~\begin{array}{ll|} 
			(\frac{n}{\alpha},\frac{2}{\alpha})&  \\		
			(\frac{n}{2},1) & 
		\end{array} ~ -{\frac{|\X-\Y|^2t^{-\frac{2}{\alpha}}}{4}e^{-\frac{i\pi \theta }{\alpha}}} \right]\right)\frac{\Y}{|\Y|^{n+1}}d\Y,
	\end{eqnarray*}
whereby $\Re(s)$ resp. $\Im(s)$ stands for the real resp. imaginary part of $s\in\BC$.\label{Statement1Theorem}
	\item In case of $\theta=0$, $\Phi_\alpha(\X,t;0)$ simplifies to
	\begin{eqnarray*}
	\Phi_\alpha(\X,t;0)=\frac{2^{1-n}}{\alpha \pi^{\frac{n}{2}}t^{\frac{n}{\alpha}}}~{~}_1\Psi_1
	\left[~\begin{array}{ll|} 
		(\frac{n}{\alpha},\frac{2}{\alpha})&  \\		
		(\frac{n}{2},1) & 
	\end{array} ~ -{\frac{|\X|^2t^{-\frac{2}{\alpha}}}{4}}\right].
	\end{eqnarray*}\label{Statement2Theorem}
	\end{enumerate}
\end{theorem}

\begin{remark}
	The fundamental solution $\Phi_{\alpha}(\X,t;\theta)$ provided by {\bf Theorem \ref{FundamentalSolutionDirac}} and {\bf Theorem \ref{FundamentalSolutionDiracSeries}} resembles to the integral representation considered by K. Gorska et al in \cite[Section 2]{gorska2013higher} to exploit the Airy integral transform, formely treated on D.V. Widder's paper \cite{widder1979airy}.

In particular, $\Phi_{2m}(\X,t;0)$ (see Statement \ref{Statement2Theorem} of {\bf Theorem \ref{FundamentalSolutionDiracSeries}}) resembles the 
	well-know \textit{symmetric L\'evy stable signed functions} $g_{2m}(u)$ (see \cite[Section 3]{gorska2013higher}).
	Interesting enough, the {\it real part of} $\Phi_{2m+1}(\X,t;\pm 1)$, that yields from the substitutions $\alpha=2m+1$ and $\theta=\pm 1$, may be considered as hypercomplex extensions of \textit{higher-order Airy functions} $g_{2m+1}^\pm(u)$ obtained in \cite[Section 4]{gorska2013higher}. 
\end{remark}

\section{Preliminaries}\label{Preliminaries}

\subsection{Clifford algebras and Dirac operators}

Let $\e_1,\e_2,\ldots,\e_n$ be an orthonormal basis of $\BR^n$ satisfying the graded anti-commuting relations
\begin{eqnarray}
\label{CliffordBasis}\e_j\e_k + \e_k\e_j =-2\delta_{jk}  &\mbox{for any}& j,k=1,2,\ldots,n,
\end{eqnarray}
and $\cl_{0,n}$ the universal Clifford algebra with 
signature $(0,n)$. 

The Clifford algebra $\cl_{0,n}$ is an associative algebra with identity $\e_\emptyset:=1$, containing $\BR$ and $\BR^n$ as subspaces.
In particular, vectors $\X=(x_1,x_2\ldots,x_n)$ and $\XI=(\xi_1,\xi_2,\ldots,\xi_n)$ of $\BR^n$ are represented in terms of the linear combinations 
\begin{eqnarray*}
	\X=\sum_{j=1}^n x_j \e_j&\mbox{resp.} & \XI=\sum_{j=1}^n \xi_j \e_j,
\end{eqnarray*}
whereas the Euclidean inner product $\langle \X,\XI \rangle$ between $\X,\XI \in \BR^n$:
\begin{eqnarray*}
	\langle \X,\XI \rangle=\sum_{j=1}^n x_j\xi_j \in \BR
\end{eqnarray*}
recasted in terms of anti-commutator $\displaystyle \langle \X,\XI \rangle=-\frac{1}{2}(\X\XI+\XI \X)$, belongs to the center of $\cl_{0,n}$. By the preceding relation it follows that the square $\XI^2$ is a real number satisfying $\XI^2=-|\XI|^2$, where $|\XI|:={\langle \XI,\XI \rangle}^{\frac{1}{2}}$ stands for the Euclidean norm in $\BR^n$.

Here we would like to stress that the linear space isomorphism given by the mapping
$\e_{j_1}\e_{j_2}\ldots \e_{j_r} \mapsto dx_{j_1}dx_{j_2}\ldots
dx_{j_r}$, with $1\leq j_1<j_2<\ldots<j_r\leq n$, enable us to establish an isomorphim between $\cl_{0,n}$ and the exterior algebra
$\displaystyle \bigwedge^* (\BR^n)=\bigoplus_{r=0}^n \bigwedge^r (\BR^n)$ (cf.~\cite[Chapter 4]{vaz2016introduction}) so that $\cl_{0,n}$ has dimension $2^n$.

In the sequel we always consider multivector functions $(\X,t)\mapsto \varPsi(\X,t)$ with values on the complexified Clifford algebra $\BC \otimes \cl_{0,n}$, written as
linear combinations in terms of the $r$-multivector basis
$\e_{j_1}\e_{j_2}\ldots \e_{j_r}$ labeled by the subsets
$J=\{j_1,j_2,\ldots,j_r\}$ of $\{ 1,2,\ldots,n\}$ i.e.
\begin{eqnarray}
\label{CliffordPsi}
\varPsi(\X,t)=\sum_{r=0}^n\sum_{|J|=r} \psi_J(\X,t) \e_J, &
	\mbox{with}~\e_{J}=\e_{j_1}\e_{j_2}\ldots \e_{j_r} &\mbox{and}~\psi_J(\X,t)\in \BC.
\end{eqnarray}

It is well know that the Dirac operator $\D$, defined as
\begin{eqnarray}
\label{DiracOp}\displaystyle \D\varPsi(\X,t)=\sum_{j=1}^{n}\e_j\partial_{x_j}\varPsi(\X,t),
\end{eqnarray}
factorizes the Euclidean Laplacian $\displaystyle \Delta=\sum_{j=1}^n \partial_{x_j}^2$, that is $\D^2=-\Delta$.
Further, induction over $m\in \BN$ shows us that
\begin{eqnarray}
\label{Dk}\D^{2m}=(-\Delta)^m &\mbox{and} & \D^{2m+1}=(-\Delta)^m\D.
\end{eqnarray}

\subsection{Background on function spaces, distributions and Fourier transform}\label{FourierTransformSection}

To extend the action of the Fourier transform  (\ref{FourierTransform}) and its inverse (\ref{FourierInverse}) by component wise application to Clifford-valued functions (\ref{CliffordPsi}) we need to consider the $\dag-${\it
	conjugation} operation $\varPsi(\X,t) \mapsto\varPsi(\X,t)^\dag$ on the \textit{complexified Clifford algebra} $\BC\otimes\cl_{0,n}$ defined as 
\begin{eqnarray}
\label{dagconjugation}
\begin{array}{lll}
(\Phi(\X,t) \varPsi(\X,t))^\dag=\varPsi(\X,t)^\dag\Phi(\X,t)^\dag \\ (\psi_J(\X,t) \e_J)^\dag =\overline{\psi_J(\X,t)}~\e_{j_r}^\dag
\ldots \e_{j_2}^\dag\e_{j_1}^\dag~~~(1\leq j_1<j_2<\ldots<j_r\leq n) \\
\e_j^\dag=-\e_j~~~(1\leq j\leq
n)
\end{array}.
\end{eqnarray}

Henceforth, the norm $\| \varPsi(\X,t)\|$ of the Clifford-valued-function (\ref{CliffordPsi}) induced by the identity $\| \varPsi(\X,t)\|^2=\varPsi(\X,t)^\dag \varPsi(\X,t)$ -- which is positive definite in the view of (\ref{dagconjugation}) and of the the anti-commuting relations (\ref{CliffordBasis})  --  allows us to define  
\begin{eqnarray}
\label{LpNorm} \|\varPsi(\cdot,t)\|_{L^p}:=\left\{\begin{array}{lll} \displaystyle \left(\int_{\BR^n} \|\varPsi(\X,t)\|^p d\X\right)^{\frac{1}{p}}, &\mbox{if} & 1\leq p<\infty \\ \ \\
\displaystyle \sup_{x \in \BR^n} \|\varPsi(\X,t)\|, &\mbox{if} & p=\infty 
\end{array}\right.
\end{eqnarray}
as the underlying norm of the {\it Bochner type spaces} $L^p(\BR^n;\BC\otimes\cl_{0,n})$ (cf.~\cite[subsection 1.2.b]{hytonen2016analysis}).

We then consider the multi-index abbreviations
\begin{eqnarray*}
	\X^{(\mu)}:=\left(x_1\right)^{\mu_1}\left(x_2\right)^{\mu_2} \ldots \left(x_n\right)^{\mu_n}, & \mbox{for}& \mu:=(\mu_1,\mu_2,\ldots,\mu_n) \in (\BN_0)^n \\
	\left(\partial_{\X}\right)^\nu:=\left(\partial_{x_1}\right)^{\nu_1}\ldots \left(\partial_{x_n}\right)^{\nu_n}, & \mbox{for}& \nu:=(\nu_1,\mu_2,\ldots,\nu_n) \in (\BN_0)^n \\
	|\mu|=\mu_1+\ldots+\mu_n,& \mbox{for}& \mu:=(\mu_1,\mu_2,\ldots,\mu_n) \in (\BN_0)^n \\
	|\nu|=\nu_1+\ldots+\nu_n,& \mbox{for}& \nu:=(\nu_1,\nu_2,\ldots,\nu_n) \in (\BN_0)^n.	
\end{eqnarray*}
 to exploit the definition of the Schwartz space $\mathcal{S}(\BR^n)$ and its dual $\mathcal{S}'(\BR^n)$ to Clifford-valued distributions in the spirit of \cite[subsections 2.4.c \& 2.4.d]{hytonen2016analysis}.

The space of \textit{rapidly decaying functions} $\varPsi(\cdot,t)$ with values on $\BC\otimes\cl_{0,n}$, denoted by $\mathcal{S}(\BR^n;\BC\otimes\cl_{0,n})$, is the Fr\'echet space defined by the family of semi-norm conditions underlying to the positive constant $M>0$:
$$ \displaystyle \sup_{\X \in \BR^n~;~|\mu|+|\nu|<M} ~\left\|\X^{(\mu)} \left(\partial_{\X}\right)^\nu \varPsi(\X,t)\right\|<\infty$$
whereas the space of $\BC\otimes\cl_{0,n}-$valued \textit{tempered distributions} $\mathcal{S}'(\BR^n;\BC\otimes\cl_{0,n})$ consists of \textit{continuous linear functionals} 
induced by the mapping $\varPsi(\cdot,t) \mapsto \langle \varPsi(\cdot,t),\Phi(\cdot,t)\rangle_{L^2}$, whereby
\begin{eqnarray}
\label{SesquilinarRn}\langle \Psi(\cdot,t),\Phi(\cdot,t
)\rangle_{L^2}:= \int_{\BR^n} \varPsi(\X,t)^\dag \Phi(\X,t) d\X
\end{eqnarray}
stands for the sesquilinear form (possibly Clifford-valued).

Notice that $\mathcal{S}(\BR^n;\BC\otimes\cl_{0,n})$ is a dense subset in $L^p(\BR^n;\BC\otimes\cl_{0,n})$ (cf.~\cite[Proposition 2.4.23.]{hytonen2016analysis}) for any $1\leq p<\infty$ (i.e. for values of $p\neq \infty$). On the other hand, the sesquilinear form (\ref{SesquilinarRn}) permits us not only to exploit the main features of the Fourier transform such as the mapping property $\mathcal{F}:L^1(\BR^n;\BC\otimes\cl_{0,n})\rightarrow L^\infty(\BR^n;\BC\otimes\cl_{0,n})$ (cf.~\cite[Definition 2.4.1]{hytonen2016analysis}) and the Plancherel theorem (cf.~\cite[Theorem 2.4.9]{hytonen2016analysis}) but also to extend the action of the Fourier transform (\ref{FourierTransform}) to an isomorphism in $\mathcal{S}'(\BR^n;\BC\otimes\cl_{0,n})$ as follows:
\begin{eqnarray*}\langle \mathcal{F}\varPsi(\cdot,t),\mathcal{F}\Phi(\cdot,t)\rangle_{L^2}=\langle \varPsi(\cdot,t),\Phi(\cdot,t)\rangle_{L^2},~~\mbox{for all}~~ & \varPsi(\cdot,t)\in \mathcal{S}(\BR^n;\BC\otimes\cl_{0,n})~\&~ &  \Phi(\cdot,t)\in \mathcal{S}'(\BR^n;\BC\otimes\cl_{0,n}).
\end{eqnarray*}

In particular, that allows us to show that for every $\varPsi(\cdot,t)$ with membership in $L^p(\BR^n;\BC\otimes\cl_{0,n})$, the family of \textit{continuous linear functionals} $\varPsi(\cdot,t) \mapsto \langle \varPsi(\cdot,t),\Phi(\cdot,t)\rangle_{L^2}$ always define an element of $\mathcal{S}(\BR^n;\BC\otimes\cl_{0,n})$ (see \cite[Example 2.4.26.]{hytonen2016analysis} for further details).

\begin{remark}\label{SchwartzClassRemark}
	Since $e^{-\tau|\XI|^\alpha}$ belongs to $\mathcal{S}'(\BR^n;\BC\otimes\cl_{0,n})$ for values of $\Re(\tau)\geq 0$ (see  {\it Remark \ref{thetaRemark2}}) one can thus infer from density arguments, depicted as above, that the kernel function $K_{\alpha,n}(\X,\tau)$ defined viz eq. (\ref{LevyDistribution}) belongs to $L^p(\BR^n;\BC\otimes\cl_{0,n})$, whenever $1\leq p<\infty$.
\end{remark}

\subsection{Background on Riesz-Hilbert-Type transforms}

Next we review the construction of the Hilbert transform-- the so-called fractional Riesz-Hilbert transform. Following the framework considered in the series of papers \cite{bernstein2016fractional,bernstein2017fractional} we derive the pseudo-differential representation of the so-called Riesz-Hilbert transform through the pseudo-differential identity
\begin{eqnarray}
\label{RieszHilbertTransform}
\D=(-\Delta)^{\frac{1}{2}}\mathcal{H},
\end{eqnarray}
where $(-\Delta)^{\frac{\alpha}{2}}=\mathcal{F}^{-1} |\XI|^{\alpha} \mathcal{F} $ denotes the fractional Riesz operator of order $\alpha$ (cf.~\cite[p.~483]{samko1993fractional}).

First, we note that in the view of the spectral identity
\begin{eqnarray*}
	\mathcal{F}(D \Psi)(\XI,t) =-i\XI \mathcal{F}\Psi(\XI,t), &~~\mbox{with}~~ & \XI=\sum_{j=1}^{n}\xi_j \e_j.
\end{eqnarray*} and of the spherical decomposition
\begin{eqnarray*}
	-i\XI=|\XI| \frac{-i\XI}{|\XI|}, &\mbox{with}&|\XI|=(-\XI^2)^{\frac{1}{2}}\neq 0
\end{eqnarray*}
 it is straightforward to see that $\mathcal{H}$ admits the pseudo-differential representation (cf.~\cite[section 4.]{bernstein2016fractional})
\begin{eqnarray}
\label{RieszHilbertPseudoDiff}
\mathcal{H}=\mathcal{F}^{-1} \frac{-i\XI}{|\XI|} \mathcal{F}
\end{eqnarray}

Next, we are going to derive the a close formula for the fractional Hilbert transform. To do so, we shall make use of  the formal series expansion
\begin{eqnarray}
\label{FractionalHilbertFormalSeries}
\exp\left( \frac{i\pi\theta}{2} \mathcal{H}\right)=\sum_{k=0}^{\infty} \frac{i^k\pi^k\theta^k}{2^kk!} \mathcal{H}^k
\end{eqnarray}

The next lemma corresponds to an abridged version of \cite[Definition 4.1]{bernstein2017fractional} (see eq. (4.2)) and \cite[Theorem 4.2.]{bernstein2016fractional} (see also \cite[Theorem 4.2.]{bernstein2017fractional}):
\begin{lemma}\label{FractionalHilbertLemma}
The fractional Riesz-Hilbert transform $\exp\left( \frac{i\pi\theta}{2} \mathcal{H}\right)$ admits the polar representation
\begin{eqnarray}
\label{FractionalHilbert}
\exp\left( \frac{i\pi\theta}{2} \mathcal{H}\right)=
\cos\left( \frac{\pi\theta}{2}\right)I+i\sin\left( \frac{\pi\theta}{2}\right)\mathcal{H}.	
\end{eqnarray}

Moreover, $$\exp\left( \frac{i\pi\theta}{2} \mathcal{H}\right):L^p(\BR^n;\BC\otimes\cl_{0,n})\rightarrow L^p(\BR^n;\BC\otimes\cl_{0,n})$$
is continuous and bounded for values of $1<p<\infty$.
\end{lemma}

\begin{proof}
From the formal series expansion (\ref{FractionalHilbertFormalSeries}) of $\exp\left( \frac{i\pi\theta}{2} \mathcal{H}\right)$
one can infer that 

\begin{eqnarray}
\label{FractionalHilbertMultiplier}\mathcal{F}\left[\exp\left( \frac{i\pi\theta}{2} \mathcal{H}\right)\Psi(\cdot,t)\right](\XI)= \exp\left( \frac{i\pi\theta}{2} h(\XI)\right)\mathcal{F}\Psi(\XI,t),
\end{eqnarray}
where $h(\XI):=\frac{-i\XI}{|\XI|}$ denotes the Fourier multiplier of $\mathcal{H}$.

Here, we recall that $h(\XI)$ arising on the exponentiation representation $\exp\left( \frac{i\pi\theta}{2} h(\XI)\right)$ is a unitary Clifford vector, i.e. $\left(~h(\XI)~\right)^{2}=1$. By expanding now $\exp\left( \frac{i\pi\theta}{2} h(\XI)\right)$ as a formal series expansion, the set of recursive relations 
\begin{eqnarray*}
\left(~h(\XI)~\right)^{2j}=1 & \mbox{and} & \left(~h(\XI)~\right)^{2j+1}=h(\XI)
\end{eqnarray*}
yielding from induction over $j$, leads to the splitting formula
\begin{eqnarray*}
\exp\left( \frac{i\pi\theta}{2} h(\XI)\right)
	&=&\sum_{j=0}^\infty \frac{(i\pi)^{2j}\theta^{2j}}{2^{2j}(2j)!} \left(~h(\XI)~\right)^{2j}+\sum_{j=0}^\infty \frac{(i\pi)^{2j+1}\theta^{2j+1}}{2^{2j+1}(2j+1)!} \left(~h(\XI)~\right)^{2j+1} \\
	&=&\sum_{j=0}^\infty \frac{\left(-1\right)^{j}\pi^{2j}\theta^{2j}}{2^{2j}(2j)!} +ih(\XI)\sum_{j=0}^\infty \frac{\left(-1\right)^{j}\pi^{2j+1}\theta^{2j+1}}{2^{2j+1}(2j+1)!}  \\
\end{eqnarray*}
which is equivalent to 
\begin{eqnarray*}
\exp\left( \frac{\pi\theta}{2} \frac{\XI}{|\XI|}\right)&=&\cos\left( \frac{\pi\theta}{2}\right)+\frac{\XI}{|\XI|}\sin\left( \frac{\pi\theta}{2}\right).
\end{eqnarray*}

Finally, by taking the inverse of the Fourier transform (\ref{FourierInverse}) on both sides of (\ref{FractionalHilbertMultiplier}) the one gets from (\ref{RieszHilbertPseudoDiff}) the eq. (\ref{FractionalHilbert}), as expected.

Moreover, the proof of continuity and boundedness of $\exp\left( \frac{i\pi\theta}{2} \mathcal{H}\right):L^p(\BR^n;\BC\otimes\cl_{0,n})\rightarrow L^p(\BR^n;\BC\otimes\cl_{0,n})$ for values of $1<p<\infty$ is, up to a linearity argument, similar to the proof of \cite[Theorem 4.2.]{bernstein2016fractional}.
\end{proof}

\section{Mellin-Barnes and Wright series representations of the Fundamental solution}\label{MellinBarnesSection}

\subsection{Mellin-Barnes representation}

Following the framework considered in \cite{gorska2013exact,gorska2013higher}, we start to to reformulate $K_{\alpha,n}(\X,\tau)$ represented through eq. (\ref{LevyDistributionSpherical}) as a Mellin convolution at the point $r=|\X|\tau^{-\frac{1}{\alpha}}$ (cf.~\cite[Section 4.]{ButJ97}) namely:
\begin{eqnarray*}
	K_{\alpha,n}(\X,\tau)&=&\frac{1}{(2\pi)^{\frac{n}{2}}|\X|^n}\int_{0}^{\infty}f\left(\frac{|\X|\tau^{-\frac{1}{\alpha}}}{\rho}\right)g(\rho)\frac{d\rho}{\rho},
\end{eqnarray*}
with 
\begin{eqnarray*}
	f(\rho)=e^{-\rho^{-\alpha}} 
	&\mbox{and}& g(\rho)=\rho^{\frac{n}{2}+1}J_{\frac{n}{2}-1}(\rho).
\end{eqnarray*}

Furthermore, by exploiting the Mellin convolution theorem (see eq.~\cite[Theorem 3.]{ButJ97}), we also have that
\begin{eqnarray}
\label{MellinFG}
	\mathcal{M}\{K_{\alpha,n}(\X,\tau)\}(s)
	&=&\frac{1}{(2\pi)^{\frac{n}{2}}|\X|^n}\mathcal{M}\left\{f\left(|\X|\tau^{-\frac{1}{\alpha}}\right)\right\}(s)~.~\mathcal{M}\left\{g\left(|\X|\tau^{-\frac{1}{\alpha}}\right)\right\}(s),
\end{eqnarray}
whereby
\begin{eqnarray}
	\label{MellinT}
	\mathcal{M}\{\phi(\rho)\}(s)=\int_{0}^{\infty} \phi(\rho)\rho^{s-1}d\rho, &\mbox{with}& s\in\BC,
\end{eqnarray} stands for the Mellin transform (cf.~\cite[eq.(1.2) of p. 326]{ButJ97}).

The next step thus consists in computing the Mellin transforms $\mathcal{M}\{f(\rho)\}(s)$ and $\mathcal{M}\{g(\rho)\}(s)$ and then inverting the product $\mathcal{M}\{f(\rho)\}(s).\mathcal{M}\{g(\rho)\}(s)$ at $\rho=|\X|\tau^{-\frac{1}{\alpha}}$ by means of the Mellin inversion formula (cf.~\cite[eq. (1.3) of p.326]{ButJ97})
\begin{eqnarray}
	\label{MellinInv}	\phi(\rho)=\frac{1}{2\pi i}\int_{c-i\infty}^{c+i\infty}\mathcal{M}\{\phi(\rho)\}(s)~\rho^{-s}~ds, & \mbox{with} & \rho>0 ~~~\&~~ c=\Re(s).
\end{eqnarray}
 Thereby, from a straightforwardly application of property \begin{eqnarray*}
	\label{MellinP}	\mathcal{M}\{\rho^\beta \phi(\kappa \rho^\gamma)\}(s)=\frac{1}{|\gamma|}\kappa^{-\frac{s+\beta}{\gamma}}\mathcal{M}\{\phi(\rho)\}\left(\frac{s+\beta}{\gamma}\right)
\end{eqnarray*}
carrying the parameters $\beta\in \BC,~\gamma \in \BC\setminus \{0\}$ and $\kappa>0$ (cf.~\cite[Proposition 1.]{ButJ97}), there holds
\begin{eqnarray*}
\mathcal{M}\{f(\rho)\}(s)&=&\frac{1}{|-\alpha|}\mathcal{M}\{e^{-\rho}\}\left(-\frac{s}{\alpha}\right)\\
&=&\frac{1}{\alpha}\Gamma\left(-\frac{s}{\alpha}\right) \\ \ \\
\mathcal{M}\{g(\rho)\}(s)&=&\mathcal{M}\left\{J_{\frac{n}{2}-1}\left({\rho}\right)\right\}\left(s+\frac{n}{2}+1\right)\\
&=&{2^{\frac{n}{2}+s}}\frac{\Gamma\left(\frac{n}{2}+\frac{s}{2}\right)}{\Gamma\left(-\frac{s}{2}\right)}
\end{eqnarray*}
so that (\ref{MellinFG}) equals
\begin{eqnarray}
\label{MellinKernel}
	\mathcal{M}\{K_{\alpha,n}(\X,\tau)\}(s)=\frac{2^s}{\alpha \pi^{\frac{n}{2}}|\X|^n}\frac{\Gamma\left(\frac{n}{2}+\frac{s}{2}\right)\Gamma\left(-\frac{s}{\alpha}\right)}{\Gamma\left(-\frac{s}{2}\right)}.
\end{eqnarray}

On computations above we have used the Eulerian representation of the Gamma function (see \cite[Chapter 13]{george2013mathematical})
\begin{eqnarray}
\label{GammaInt}\int_{0}^{\infty} e^{-\rho}\rho^{s-1}d\rho=\Gamma(s)
\end{eqnarray}
and the Weber integral representation (cf.~\cite[p. 490, eq. (25.27)]{samko1993fractional}) carrying the parameters $\beta=\frac{n}{2}+s$ and $\nu=\frac{n}{2}-1$ (see also \cite[p.~57, eq. (2.46)]{mathai2009h}):
\begin{eqnarray}
\label{WeberIntegral}
\int_{0}^{\infty} \rho^\beta J_\nu(\rho) d\rho=2^\beta~ \frac{\Gamma\left(\frac{\nu+\beta+1}{2}\right)}{\Gamma\left(\frac{\nu-\beta+1}{2}\right)},& \mbox{for}& -\Re(\nu)-1<\Re(\beta)<\frac{1}{2}
\end{eqnarray}

As a consequence of (\ref{MellinFG}) and of the Mellin inversion formula (\ref{MellinInv}) we conclude that for $\rho=|\X|\tau^{-\frac{1}{\alpha}}$
\begin{eqnarray}
\label{LevyMellinBarnes}
	\begin{array}{lll}
		K_{\alpha,n}(\X,\tau)&=&\displaystyle \frac{1}{2\pi i}\int_{c-i\infty}^{c+i\infty} 	\mathcal{M}\{K_{\alpha,n}(\X,\tau)\}(s)~ \left(|\X|\tau^{-\frac{1}{\alpha}}\right)^{-s} ds\\ \ \\
		&=&\displaystyle \frac{1}{\alpha \pi^{\frac{n}{2}}|\X|^n}~\frac{1}{2\pi i}\int_{c-i\infty}^{c+i\infty} \frac{\Gamma\left(\frac{n}{2}+\frac{s}{2}\right)\Gamma\left(-\frac{s}{\alpha}\right)}{\Gamma\left(-\frac{s}{2}\right)}~ \left({\frac{|\X|\tau^{-\frac{1}{\alpha}}}{2}}\right)^{-s}ds,
	\end{array}
\end{eqnarray} 
or equivalently,
$$
K_{\alpha,n}(\X,\tau)=\frac{1}{\alpha \pi^{\frac{n}{2}}|\X|^n}H^{1,1}_{1,2}
\left[~{\frac{|\X|\tau^{-\frac{1}{\alpha}}}{2}} ~\begin{array}{|ll} 
(1,\frac{1}{\alpha})&  \\ \ \\ 	
(\frac{n}{2},\frac{1}{2}) & (1,\frac{1}{2}) 
\end{array} ~ \right]
$$
using the H-function notation (\cite[p.~2]{mathai2009h} \& \cite[eq. (5.1)]{kilbas2002generalized}).

By applying general existence conditions for the H-function  we conclude that, in case of $\alpha>1$, the integral representation (\ref{LevyMellinBarnes}) yields a uniformly convergent series expansion -- conditions $\mu>0$ and $q\geq 1$ carrying the parameters $\mu=\frac{1}{2}+\frac{1}{2}-\frac{1}{\alpha}$ and $q=2$ (cf.~\cite[Case 1: of Theorem 1.1.]{mathai2009h}).

\subsection{Series Expansion of Wright type}

To represent explicitly $K_{\alpha,n}(\X,\tau)$ as a convergent series expansion of Wright type, one transform the closed path joining the endpoints $c-i\infty$ and $c+i\infty$ associated to the straight line $\Re(s)=c$ ($-n<c<\frac{1-n}{2}$) as a loop beginning and ending at $-\infty$ and enriching all the simple poles $s=-n-2k$ ($k\in\BN_0$).

Using the fact that for values of $\alpha>1$ the intersection of the simple poles of  $\Gamma\left(\frac{n}{2}-s\right)$ and $\Gamma\left(-\frac{s}{\alpha}\right)$ yields an empty set ($\alpha k\neq -n-2k$ for all $k\in\BN_0$), there holds  by a straightforward aplication of the standard residue theorem, we can evaluate the contour integral appearing on the right-hand side of (\ref{LevyMellinBarnes}) as an infinite sum of the residues at $s=-n-2k$ ($k\in\BN_0$).

 Namely, we have
\begin{eqnarray*}
	K_{\alpha,n}(\X,\tau)&=&\frac{1}{\alpha \pi^{\frac{n}{2}}|\X|^n}\sum_{k=0}^\infty \lim_{s\rightarrow
		-n-2k}(s+n+2k) \frac{\Gamma\left(\frac{n}{2}+\frac{s}{2}\right)\Gamma\left(-\frac{s}{\alpha}\right)}{\Gamma\left(-\frac{s}{2}\right)}~ \left({\frac{|\X|\tau^{-\frac{1}{\alpha}}}{2}}\right)^{-s} \\
	&=& 
	\frac{1}{\alpha \pi^{\frac{n}{2}}|\X|^n}\sum_{k=0}^\infty \lim_{s\rightarrow
		-n-2k}
	\left(\frac{n}{2}+\frac{s}{2}+k\right)\Gamma\left(\frac{n}{2}+\frac{s}{2}\right)~\frac{2\Gamma\left(-\frac{s}{\alpha}\right)}{\Gamma\left(-\frac{s}{2}\right)}~ \left({\frac{|\X|^2\tau^{-\frac{2}{\alpha}}}{4}}\right)^{-\frac{s}{2}}   \\
	&=&	\frac{1}{\alpha \pi^{\frac{n}{2}}|\X|^n}\sum_{k=0}^\infty \frac{(-1)^k}{k!} ~ ~\frac{2\Gamma\left(\frac{n}{\alpha}+\frac{2}{\alpha}k\right)}{\Gamma\left(\frac{n}{2}+k\right)}~ \left({\frac{|\X|^2\tau^{-\frac{2}{\alpha}}}{4}}\right)^{\frac{n}{2}+k}.
\end{eqnarray*}

After a straightforward simplification we recognize that, for values of $\alpha>1$, the kernel function $K_{\alpha,n}(\X,\tau)$ reduces to a Wright series expansion of the type (\ref{WrightSeriespq}). Namely, one gets the closed representation
\begin{eqnarray}
\label{LevyDistributionWright11}
	K_{\alpha,n}(\X,\tau)=\frac{2^{1-n}}{\alpha \pi^{\frac{n}{2}}\tau^{\frac{n}{\alpha}}}~{~}_1\Psi_1
\left[~\begin{array}{ll|} 
(\frac{n}{\alpha},\frac{2}{\alpha})&  \\		
(\frac{n}{2},1) & 
\end{array} ~ -{\frac{|\X|^2\tau^{-\frac{2}{\alpha}}}{4}} \right].
\end{eqnarray}

\begin{remark}\label{RemarkStatements123}
	From the substitution $\tau=te^{i\frac{\pi\theta}{2}}$ on (\ref{LevyDistributionWright11}). it thus follows then
	\begin{eqnarray*}
		\displaystyle K_{\alpha,n}(\X,te^{i\frac{\pi\theta}{2}})&=&\frac{2^{1-n}}{\alpha \pi^{\frac{n}{2}}t^{\frac{n}{\alpha}}}~e^{-\frac{i\pi \theta }{\alpha}\frac{n}{2}}~{~}_1\Psi_1
		\left[~\begin{array}{ll|} 
			(\frac{n}{\alpha},\frac{2}{\alpha})&  \\		
			(\frac{n}{2},1) & 
		\end{array} ~ -{\frac{|\X|^2\tau^{-\frac{2}{\alpha}}}{4}}e^{-\frac{i\pi \theta }{\alpha}} \right]\\
		&=&\frac{2^{1-n}}{\alpha \pi^{\frac{n}{2}}t^{\frac{n}{\alpha}}}\sum_{k=0}^\infty  ~ ~\frac{\Gamma\left(\frac{n}{\alpha}+\frac{2}{\alpha}k\right)}{\Gamma\left(\frac{n}{2}+k\right)}~ \frac{\left(-{\frac{|\X|^2\tau^{-\frac{2}{\alpha}}}{4}}\right)^{k}e^{-i\frac{\pi \theta}{\alpha}\left(\frac{n}{2}+k\right)}}{k!}.
	\end{eqnarray*}
\end{remark}

\begin{remark}
	For $\alpha=2$ and $\tau=t$ one has that $	K_{2,n}(\X,t)$, determined as above, simplifies to 
	$$K_{2,n}(\X,t)=\frac{1}{\left(4\pi t\right)^{\frac{n}{2}}}~e^{-{\frac{|\X|^2}{4t}}},$$
	the so-called \textit{heat kernel}.
\end{remark}

\section{Proof of the Main Results}\label{ProofMainResults}

\subsection{Proof of Theorem \ref{FundamentalSolutionDirac}}\label{ProofMainResults1}

To proof {\bf Theorem \ref{FundamentalSolutionDirac}} one has to consider, as in the refs. \cite{li1994clifford,mcintosh1996clifford}, the projection operators $\chi_\pm(\XI)$ ($\XI \neq 0$) defined viz 
\begin{eqnarray}
\label{ProjectionOp}
	\chi_\pm(\XI)=\frac{1}{2}\left(1\pm  \frac{i\XI}{|\XI|}\right)
\end{eqnarray}
corresponding to the Fourier multipliers of $\frac{1}{2}(I\mp\mathcal{H}):L^p(\BR^n;\BC\otimes\cl_{0,n})\rightarrow L^p(\BR^n;\BC\otimes\cl_{0,n})$ ($1<p<\infty$). See also 

Indeed, the set of properties 
\begin{eqnarray}
\label{ProjectionOpProperties}
	\chi_+(\XI)+\chi_-(\XI)=1,~~ & \left(\chi_\pm(\XI)\right)^2=\chi_\pm(\XI) &~~\mbox{and}~~\chi_+(\XI)\chi_-(\XI)=0=\chi_-(\XI)\chi_+(\XI)
\end{eqnarray}
yield straightforwardly from the fact that the Fourier multiplier $h(\XI):=\frac{-i\XI}{|\XI|}$ of $\mathcal{H}$ is a unitary Clifford-vector, i.e. $\left(~h(\XI)~\right)^{2}=1$ (cf.~\cite[p. 671]{li1994clifford} \& \cite[subsection 5.2.1]{mcintosh1996clifford}), whereas in the view of (\ref{ProjectionOp}) and of (\ref{RieszHilbertPseudoDiff})
\begin{eqnarray}
\label{ProjectionOpHilbert}
\begin{array}{ccc}
\frac{1}{2}(I+\mathcal{H})&=&\mathcal{F}^{-1} \chi_-(\XI) \mathcal{F} \\ \frac{1}{2}(I-\mathcal{H})&=&\mathcal{F}^{-1} \chi_+(\XI)  \mathcal{F}.
\end{array}
\end{eqnarray} 

\begin{proof}({\bf Theorem \ref{FundamentalSolutionDirac}})
First note that in the view of {\bf Lemma \ref{FractionalHilbertLemma}}, the Cauchy problem (\ref{FractionalDirac}) on the Fourier space takes the form
\begin{eqnarray*}
\left\{\begin{array}{lll}
 \partial_t	\mathcal{F}(\Phi_\alpha(\cdot,t;\theta))(\XI)=-|\XI|^\alpha\left(\cos\left( \frac{\pi\theta}{2} \right)+\frac{\XI}{|\XI|}\sin\left( \frac{\pi\theta}{2} \right)\right)	\mathcal{F}(\Phi_\alpha(\cdot,t;\theta))(\XI)&,~\mbox{for} & (\XI,t)\in \BR^n \times (0,\infty) \\ \ \\
\mathcal{F}(\Phi_\alpha(\cdot,t;\theta))(\XI)=1 &,~\mbox{for} &\XI \in \BR^n
\end{array}\right.
\end{eqnarray*}
which, in terms of the projection operators $\chi_\pm(\XI)$ defined viz eq. (\ref{ProjectionOp}), may be reformulated as 

\begin{eqnarray}
\label{FractionalDiracFourier}
\left\{\begin{array}{lll}
	\partial_t	\mathcal{F}(\Phi_\alpha(\cdot,t;\theta))(\XI)=-|\XI|^\alpha\left(e^{i\frac{\pi\theta}{2}}\chi_-(\XI)+e^{-i\frac{\pi\theta}{2}}\chi_+(\XI)\right)&,~\mbox{for} & (\XI,t)\in \BR^n \times (0,\infty) \\ \ \\
\mathcal{F}(\Phi_\alpha(\cdot,t;\theta))(\XI)=1 &,~\mbox{for} &\XI \in \BR^n
\end{array}\right.
\end{eqnarray}

Thus, a solution of (\ref{FractionalDiracFourier}) is provided, for values of $t\geq 0$, by the formal exponentiation formula
\begin{eqnarray}
\label{FractionalDiracFourierFormalSolution}
\begin{array}{ccc}
\mathcal{F}(\Phi_\alpha(\cdot,t;\theta))&=&\exp\left(-t|\XI|^\alpha \left(e^{i\frac{\pi\theta}{2}}\chi_-(\XI)+e^{-i\frac{\pi\theta}{2}}\chi_+(\XI)\right)\right) \\
&=&\displaystyle \sum_{k=0}^\infty \frac{(-1)^kt^k |\XI|^{\alpha k}}{k!}\left(e^{i\frac{\pi\theta}{2}}\chi_-(\XI)+e^{-i\frac{\pi\theta}{2}}\chi_+(\XI)\right)^k.
\end{array}
\end{eqnarray}

Application of the projection properties (\ref{ProjectionOpProperties}) carrying the idempotents $\chi_\pm(\XI)$ results, after a straightforwardly computation based on induction arguments, into the identity
\begin{eqnarray*}
\left(e^{i\frac{\pi\theta}{2}}\chi_-(\XI)+e^{-i\frac{\pi\theta}{2}}\chi_+(\XI)\right)^k=e^{i\frac{\pi\theta k}{2}}\chi_-(\XI)+e^{-i\frac{\pi\theta k}{2}}\chi_+(\XI), &\mbox{for all}&k\in\BN_0.
\end{eqnarray*}

Thus, (\ref{FractionalDiracFourierFormalSolution}) simplifies to
\begin{eqnarray*}
	\mathcal{F}(\Phi_\alpha(\cdot,t;\theta)) 
	&=& \sum_{k=0}^\infty \frac{(-1)^kt^k |\XI|^{\alpha k}}{k!}\left(e^{i\frac{\pi\theta k}{2}}\chi_-(\XI)+e^{-i\frac{\pi\theta k}{2}}\chi_+(\XI)\right) \\
	&=&\chi_-(\XI)\exp\left(-te^{i\frac{\pi\theta}{2}}|\XI|^\alpha \right)+\chi_+(\XI)\exp\left(-te^{-i\frac{\pi\theta}{2}}|\XI|^\alpha \right).
\end{eqnarray*}

Now taking the inverse $\mathcal{F}^{-1}$ of the Fourier transform (see eq. (\ref{FourierInverse})) on both sides of the preceding equation, there holds from  eqs. (\ref{LevyDistribution}) and (\ref{ProjectionOpHilbert})
\begin{eqnarray*}
\Phi_\alpha(\X,t;\theta)&=&\frac{1}{(2\pi)^n}\int_{\BR^n} \chi_-(\XI)\exp\left(-te^{i\frac{\pi\theta}{2}}|\XI|^\alpha \right)~e^{i \langle \X,\XI \rangle}d\XI+\frac{1}{(2\pi)^n}\int_{\BR^n} \chi_+(\XI)\exp\left(-te^{-i\frac{\pi\theta}{2}}|\XI|^\alpha \right)~e^{i \langle \X,\XI \rangle}d\XI \\
&=&\frac{1}{2}(I+\mathcal{H})K_{\alpha,n}(\X,te^{i\frac{\pi\theta}{2}})+\frac{1}{2}(I-\mathcal{H})K_{\alpha,n}(\X,te^{-i\frac{\pi\theta}{2}}),
\end{eqnarray*}
as desired.
\end{proof}

\begin{remark}[(see also Remark \ref{thetaRemark2} )]\label{thetaRemark}
	By combining {\it Remark \ref{SchwartzClassRemark}} and {\bf Lemma \ref{FractionalHilbertLemma}} one can easily infer that the components of $\Phi_\alpha(\X,t;\theta)$, say $\frac{1}{2}(I\pm \mathcal{H})K_{\alpha,n}(\X,te^{\pm i\frac{\pi\theta}{2}})$, are well defined distributions with membership in the {\it Bochner type spaces} $L^p(\BR^n;\BC\otimes\cl_{0,n})$ in case where we restrict to values of $p$ in the range $1< p<\infty$ ($p\neq 1$ and $p\neq \infty$).
\end{remark}

\begin{remark}[(see also Remark \ref{LiWongRemark}.)]
	In the view of the set of identities 
	\begin{eqnarray*}
		\D^{2m}= (-\Delta)^{\frac{2m}{2}}\exp\left( \frac{i\pi0}{2} \mathcal{H}\right)& \mbox{and}&
		\pm i\D^{2m+1}=(-\Delta)^{\frac{2m+1}{2}}\exp\left( \pm \frac{i\pi1}{2} \mathcal{H}\right)
	\end{eqnarray*}
	that yield straightforwardly from eq. (\ref{Dk}) and from {\bf Lemma \ref{FractionalHilbertLemma}} one can say that the fundamental solutions $\Phi_\alpha(\X,t;\theta)$ of (\ref{FractionalDirac}) encompasses the fundamental solutions $K_{2m,n}(\X,t)$ (see eq. \ref{LevyDistribution}) of the polyharmonic heat operator $\partial_t+(-\Delta)^m$ ($\Phi_{2m}(\X,t;0)=K_{2m,n}(\X,t)$) as well as the fundamental solutions $\Phi_{2m+1}(\X,t;-1)$ and $\Phi_{2m+1}(\X,t;1)$ of the higher-order Dirac-type operators $\partial_t- i \D^{2m+1}$ and $\partial_t+ i \D^{2m+1}$, respectively.
\end{remark}

\subsection{Proof of Theorem \ref{FundamentalSolutionDiracSeries}}\label{ProofMainResults2}

To prove {\bf Theorem \ref{FundamentalSolutionDiracSeries}} we will use the fact that the Riesz-Hilbert transform $\mathcal{H}$, as represented through eq. (\ref{RieszHilbertPseudoDiff}), may be expressed as a linear combination involving the Riesz operators $R_j=\mathcal{F}^{-1}\frac{-i\xi_j}{|\XI|}  \mathcal{F}$:
\begin{eqnarray}
\label{RieszHilbertLinearCombination}\mathcal{H}=\sum_{j=1}^{n}\e_j R_j.
\end{eqnarray}
on which each $R_j$ ($j=1,2,\ldots,n$) is a singular integral operator uniquely determined by the kernel functions 
\begin{eqnarray}
\label{RieszKernel}
E_j(x)=\frac{\Gamma\left(\frac{n+1}{2}\right)}{\pi^{\frac{n+1}{2}}}\frac{x_j}{|\X|^{n+1}}.
\end{eqnarray}

In particular, the spectral property (cf.~\cite[p.~224]{stein1971introduction})
\begin{eqnarray}
\label{RieszKernelFourier}
(\mathcal{F}E_j)(\XI)=\frac{-i\xi_j}{|\XI|}
\end{eqnarray}
together with framework developed in Section \ref{MellinBarnesSection} will be of foremost interest to derive a closed formula for $\Phi_\alpha(\X,t;\theta)$.

\begin{proof}	
	
	\paragraph{(Theorem \ref{FundamentalSolutionDiracSeries})}
	
	\paragraph{\it Proof of Statement \ref{Statement1Theorem}.}
	
	First, recall that in the view of \cite[p.225, Theorem 3.1]{stein1971introduction} and of the identity (\ref{RieszKernelFourier}) involving the Riesz kernels (\ref{RieszKernel}) one can easily infer, by linearity arguments, that is a convolution-type operator represented in $L^p(\BR^n;\cl_{0,n})$ ($1<p<\infty$)
	through the singular integral operator identity
	\begin{eqnarray}
	\label{RieszHilbertSingularInt}
	\mathcal{H} \varPsi(\X,t)=\frac{\Gamma\left(\frac{n+1}{2}\right)}{\pi^{\frac{n+1}{2}}} P.V. \int_{\BR^n} \varPsi(\X-\Y,t)\frac{y}{|y|^{n+1}}d\Y, &\mbox{with} & \Y=\sum_{j=1}^n y_j\e_j
	\end{eqnarray}
	
	Then, from an easy algebraic manipulation one gets that $\Phi_\alpha(\X,t;\theta)$, as obtained in {\bf Theorem \ref{FundamentalSolutionDirac}}, admits the singular integral representation
	\begin{eqnarray}
	\label{PhiSingularInt}
	\Phi_\alpha(\X,t;\theta)&=&\Re K_{\alpha,n}(\X,te^{i\frac{\pi\theta}{2}})+i\frac{\Gamma\left(\frac{n+1}{2}\right)}{\pi^{\frac{n+1}{2}}} P.V. \int_{\BR^n} \Im K_{\alpha,n}(\X-\Y,te^{i\frac{\pi\theta}{2}})\frac{\Y}{|\Y|^{n+1}}d\Y,
	\end{eqnarray}
with
\begin{eqnarray*}
\Re K_{\alpha,n}(\X,te^{i\frac{\pi\theta}{2}})&=& \frac{K_{\alpha,n}(\X,te^{i\frac{\pi\theta}{2}})+K_{\alpha,n}(\X,te^{-i\frac{\pi\theta}{2}})}{2}, \\
\Im K_{\alpha,n}(\X-\Y,te^{i\frac{\pi\theta}{2}})&=& \frac{K_{\alpha,n}(\X-\Y,te^{i\frac{\pi\theta}{2}})-K_{\alpha,n}(\X-\Y,te^{-i\frac{\pi\theta}{2}})}{2i}. 
\end{eqnarray*}

Thus, in the view of Wright series representation obtained in {\it Remark \ref{RemarkStatements123}}, the {\it proof of Statement \ref{Statement1Theorem}} is then immediate. 

\paragraph{\it Proof of Statement \ref{Statement2Theorem}.}

For the proof of {\it Statement \ref{Statement2Theorem}} of {\bf Theorem \ref{FundamentalSolutionDiracSeries}}
we note that for $\theta=0$ there holds 
\begin{center}
$\Re K_{\alpha,n}(\X,t)=K_{\alpha,n}(\X,t)$ and $\Im K_{\alpha,n}(\X-\Y,t)=0$
\end{center}
 so that (\ref{PhiSingularInt}) simplifies to 	\begin{eqnarray*}
 	\Phi_\alpha(\X,t;0)=\frac{2^{1-n}}{\alpha \pi^{\frac{n}{2}}t^{\frac{n}{\alpha}}}~{~}_1\Psi_1
 	\left[~\begin{array}{ll|} 
 		(\frac{n}{\alpha},\frac{2}{\alpha})&  \\		
 		(\frac{n}{2},1) & 
 	\end{array} ~ -{\frac{|\X|^2\tau^{-\frac{2}{\alpha}}}{4}}\right] &\mbox{[~case of $\tau=t$ in eq. (\ref{LevyDistributionWright11})~]}.
 \end{eqnarray*}

\end{proof}

\section*{Acknowledgments}

This research was supported by The Center for Research and Development in Mathematics and Applications (CIDMA) through the Portuguese Foundation for Science and Technology (FCT), references UIDB/04106/2020 and UIDP/04106/2020.

\subsection*{Conflict of interest}

The author declare no potential conflict of interests.

\bibliography{ToKlaus_NelsonFaustino_arXiv}%

\end{document}